\documentclass[11pt]{article}
\usepackage{latexsym,amsmath,color,amsthm,amssymb,epsfig,graphicx,mathrsfs}
\usepackage{graphicx}
\usepackage{amssymb}
\usepackage[left=0.9in,top=0.9in,right=0.9in,bottom=0.9in]{geometry}
\usepackage[linktocpage=true]{hyperref}
\usepackage{setspace}
\usepackage{amssymb, amsmath, amsthm, graphicx,mathrsfs}
\usepackage{caption}

\def\qed{\hfill\ifhmode\unskip\nobreak\fi\quad\ifmmode\Box\else\hfill$\Box$\fi}
\def\ite#1{\hfill\break${}$\hbox to 50pt {\quad(#1)\hfill}}
\def\eps{{\varepsilon}}
\def\phii{{\varphi}}
\newtheorem{thm}{Theorem}[section]

\newtheorem{rem}[thm]{Remark}
\newtheorem{lem}[thm]{Lemma}
\newtheorem{conj}[thm]{Conjecture}

\newtheorem{prop}[thm]{Proposition}%[section]

\def\ex{{\rm{ex}}}

\begin{document}

\title{ Hypergraphs not containing a tight tree with a bounded trunk}

\pagestyle{myheadings} \markright{{\small{\sc F\"uredi, Jiang, Kostochka, Mubayi, Verstra\"ete:  Tur\'an numbers of tight trees}}}

\author{
\hspace{0.8in} Zolt\'an F\" uredi\thanks{Research supported by grant K116769
from the National Research, Development and Innovation Office NKFIH, and
by the Simons Foundation Collaboration grant \#317487.}
\and
Tao Jiang\thanks{Research partially supported by NSF award DMS-1400249.}
\and
Alexandr Kostochka\thanks{Research of this author is supported in part by NSF grant
 DMS-1600592 and by grants 15-01-05867  and 16-01-00499 of the Russian Foundation for Basic Research.
} \hspace{0.8in} \and
Dhruv Mubayi\thanks{Research partially supported by NSF award DMS-1300138.} \and Jacques Verstra\"ete\thanks{Research supported by NSF award DMS-1556524.}
}

\date{ December 10,  2017} % \today}

\maketitle

\vspace{-0.3in}

\begin{abstract}
An $r$-uniform hypergraph is a  {\it  tight $r$-tree} if
 its edges can be  ordered so that every
edge $e$ contains a vertex $v$ that does not belong to any preceding edge
and the set $e-v$ lies in some preceding edge.
%its edges can be ordered as $e_1,\dots, e_t$ so
%that for each $i\geq 2$, there are a vertex $v\in e_i$ and $1\leq s\leq i-1$ such that $v\notin \bigcup_{j=1}^{i-1}e_j$ and
%$e_i-v\subset e_s$.
A  conjecture of Kalai~\cite{Kalai}, generalizing the Erd\H{o}s-S\'os Conjecture  for trees, asserts that
 if $T$ is a tight $r$-tree with $t$ edges and $G$ is an $n$-vertex  $r$-uniform hypergraph containing no copy of $T$ then $G$ has at most
 $\frac{t-1}{r}\binom{n}{r-1}$ edges.

A {\em trunk} $T'$ of a tight $r$-tree $T$ is a tight subtree such that every edge of $T-T'$ has $r-1$ vertices in some edge of $T'$
and a vertex outside $T'$. For $r\ge 3$, the only nontrivial family of tight $r$-trees for which this conjecture has been proved  is 
 the family of $r$-trees with trunk size one  in~\cite{FF} from 1987.
  Our main result is an asymptotic version of Kalai's conjecture for all tight trees $T$ of bounded
trunk size.   This follows from { our} upper bound   on the size of a $T$-free $r$-uniform hypergraph $G$ in terms of the size of its shadow.
%However, we observe that this is equivalent to the same result about the number of edges, not only for this question but also for
%some other questions in extremal set theory.   
We also give a short proof of Kalai's conjecture for  tight $r$-trees with at most four edges. In particular, for $3$-uniform
hypergraphs, our result on the tight path of length $4$ implies the intersection shadow theorem of Katona \cite{Katona}.
%The case of  the  3-uniform tight path with five edges remains open. 
%shows that many of the shadow theorems in the literature %are not really stronger than their nonshadow versions.

\end{abstract}

\section{
Results and history of tight trees}

%Introduction}
 Tur\'{a}n-type problems are among central in combinatorics. For  integers $n\geq r\geq 2$ and an
$r$-uniform hypergraph ({\em $r$-graph}, for short) $H$, the {\em Tur\'an number} $\ex_r(n,H)$ is the largest
$m$ such that there exists an $n$-vertex $r$-graph $G$ with $m$ edges that does not contain $H$.
One of  well-known conjectures in extremal graph theory is the Erd\H{o}s-S\'os Conjecture (see \cite{Erdos-Sos-Tree})
that {\em every $n$-vertex graph $G$ with more than $n(t-1)/2$ edges contains every tree with $t$ edges as a subgraph}. In other
words, they conjecture that $\ex_2(n,T)\leq n(t-1)/2$ for each tree with $t$ edges.
The conjecture, if true, would be best possible whenever $t$ divides $n$, as seen by taking $G$ to be the disjoint union of $K_t$'s.
There are many partial results on the conjecture. The most significant progress on the conjecture was made by Ajtai, Koml\'os,
Simonovits, and Szemer\'edi~\cite{AKSS}, who solved the conjecture for all sufficiently large $t$.

In 1984, Kalai~\cite{Kalai} made a more general conjecture for $r$-graphs.
To describe the conjecture, we need the following notion of  hypergraph trees.
Let $r\geq 2$ be an integer.
An $r$-graph $T$ is called a tight {\it $r$-tree} if its edges can be ordered as $e_1,\dots, e_t$ so that
\begin{equation}\label{tree-definition}
\parbox{5.6in}{\em for each $i\geq 2$, there are a vertex $v\in e_i$ and $1\leq s\leq i-1$ such that $v\notin \bigcup_{j=1}^{i-1}e_j$ and
$e_i-v\subset e_s$.}
\end{equation}

Note that a graph tree is a tight $2$-tree. We write $e(H)$ for the number of edges in $H$.
\begin{conj}[Kalai 1984, see in~\cite{FF}] \label{kalai}
Let $r\geq 2$ and let $T$ be a tight $r$-tree with $t\ge 2$ edges.
%Every  $n$-vertex $r$-graph $G$ that does not contain $T$ %as a subgraph
% has at most
Then $\ex_r(n,T)\leq \frac{t-1}{r}\binom{n}{r-1}$.
\end{conj}

Kalai observed that his conjecture, if true,  is asymptotically optimal using constructions obtained from  partial Steiner systems due to R\"odl~\cite{Rodl}.
The recent work of Keevash~\cite{Keevash}  (see also~\cite{GKLO}) on the existence of designs show that in fact
 for every $r\geq 2$ and $t$ there are infinitely many $n$ for which
there is an $n$-vertex $r$-graph $G$ with $e(G)=\frac{t-1}{r}\binom{n}{r-1}$ that  contains none of the tight $r$-trees with $t$ edges. 
For example, this bound can be achieved  for all $n> n_0(r,t)$ when some divisibility properties
hold, e.g., $n-r+2$ is divisible by $(t+r-1)!$.
This gives a lower bound $ \frac{t-1}{r}\binom{n}{r-1}-O_{r,t}(n^{r-2})$ for all $n$.

A weaker upper bound
\begin{equation}\label{tb}
 \ex_r(n,T)\leq {(e(T)-1)}\binom{n}{r-1}\quad\qquad\mbox{\em for each tight $r$-tree $T$}
\end{equation}
is implicit in several earlier works, and is explicit in~\cite{FJ-tree} (see Proposition 5.4 there).

To prove Conjecture~\ref{kalai}, we need to improve {   the bound in}~\eqref{tb}
 by a factor of $r$.
This turns out to be  difficult even for very special cases of tight trees.
It is only recently that the  authors~\cite{FJMKV}  were able to  improve~\eqref{tb}
in the case $T$ is the tight $r$-uniform path  with $t$ edges
 by a factor of $1-1/r$.
(For short paths, $t< (3/4)r$, Patk\'os~\cite{Patkos} proved better coefficients).

So far, the only family of tight trees for which Kalai's conjecture is verified  is the family of so-called {\it star-shaped} trees.
A tight $r$-tree $T$ is {\it star-shaped} if it contains an edge $e_0$ such that $ |e\cap e_0|=r-1$ for each $ e\in T\setminus \{e_0\} \,$.

\begin{thm} [% Frankl-F\"uredi~
\cite{FF}] \label{FF}
Let $n,r,t\geq 2$ be integers. Let $G$ be an $n$-vertex $r$-graph with $e(G)>\frac{t-1}{r}\binom{n}{r-1}$. Then $G$ contains
every star-shaped tight $r$-tree with $t$ edges.
\end{thm}

Given a tight $r$-tree $T$ and a tight subtree $T'$ of $T$, we say that $T'$ is a {\it trunk} of $T$ if
there exists an edge-ordering of $T$ satisfying \eqref{tree-definition} such that the edges of
$T'$ are listed first and
for each $e\in E(T)\setminus E(T')$ there exists $e'\in E(T')$ such that $|e\cap e'|=r-1$.
Let $c(T)$ be the minimum number of edges in a trunk of $T$.
Hence, a star-shaped tight tree is a tight tree $T$ with $c(T)=1$, and Theorem~\ref{FF} says that Kalai's Conjecture
holds for tight $r$-trees $T$ with $c(T)=1$. Note from the definition above that for a tight tree $T$ having $c(T)\leq c$
is equivalent to saying that all but at most $c$ edges of $T$ contain a vertex of degree $1$.

{%  
The primary goal of this paper is to extend  
Theorem~\ref{FF} to tight trees of bounded trunk size.
Our main theorem says that for every fixed integers $r\geq 2$ and $c\geq 1$, Kalai's Conjecture holds asymptotically
in $e(T)$ for tight $r$-trees $T$ with $c(T)\leq c$.} 

%is the following, which substantially generalizes Theorem~\ref{FF}.

\begin{thm} \label{main}
Let $n,r,t,c$ be positive integers, where $n\geq r\geq 2$ and $t\geq c\geq 1$.
Let $a(r,c)=(r^r+1-\frac{1}{r})(c-1)$.
Let $T$ be a tight $r$-tree with $t$ edges and $c(T)\leq c$. % that has trunk $T'$ size at most $c$.
Then
\begin{equation}\label{eq:main}
\ex_r(n,T)\leq  \left(\frac{t-1}{r}+ a(r,c)\right)\binom{n}{r-1}.
   \end{equation}
%If $G$ is an $n$-vertex $r$-graph that does not contain $T$, then
%\[e(G)\leq \left(\frac{t-1}{r}+a\right)|\partial (G)|.\]
%In particular,
%$e(G)\leq   (\frac{t-1}{r}+(r^r+1-\frac{1}{r}) c)\binom{n}{r-1}$.
\end{thm}
Note that Theorem \ref{FF} follows from Theorem \ref{main} by setting $c=1$.
{   The} main point of Theorem \ref{main} is that the coefficient in front of $\binom{n}{r-1}$
is $(t-1)/r +O_{r,c}(1)$, while the coefficient in Kalai's conjecture is $(t-1)/r$.

We also give a (simple) proof of the fact that Kalai's Conjecture holds for tight $r$-trees with at most four edges.
\begin{thm} \label{4e}
Let $n\geq r\geq 2$ be  integers and $T$ be a tight $r$-tree with $t\leq 4$ edges.
%that has a trunk $\{e_1,e_2\}$ of size $2$.
%Let $G$ be an $n$-vertex $3$-graph that does not
%contain $T$. Then $e(G)\leq \frac{t-1}{3}|\partial(G)|$. In particular,
Then $$\ex_r(n,T)\leq \frac{t-1}{r} \binom{n}{r-1}.$$
\end{thm}

The proofs of Theorems~\ref{main} and~\ref{4e} are postponed to Sections~\ref{sec:main} and~\ref{sec:4e}.

%%%%%%%%%%%%%%%%%%%%%%%%%%%%%%%%%%%%%%%%%%%%%%%%%%%%%%%

\section{Tight trees and shadows}

An important notion in extremal set theory is that of {\em shadow}. % and {\em weights}.
Given an $r$-graph $G$,  the {\it shadow} of $G$ is
\[\partial(G)=\{S: |S|=r-1,\quad\mbox{\em and} \quad S\subseteq e \quad\mbox{\em for some}\quad   e\in e(G)\}.\]

%Theorem~\ref{FF} of
% Frankl and F\"uredi and our Theorems~\ref{main}--\ref{4e} have stronger versions in terms of shadows.
%We state and discuss them in the next section.

The result % of  Frankl and F\"uredi
 in~\cite{FF} is more explicit than Theorem~\ref{FF}.
%  They
It was shown that if $T$ is any star-shaped tight $r$-tree with $t$ edges
and $G$ is a $T$-free $r$-graph then $e(G)\leq \frac{t-1}{r}|\partial(G)|$, from which Theorem~\ref{FF} immediately follows.
There were several other results in the literature that
bound the size of an $H$-free $r$-graph in terms of the size of its shadow. Katona~\cite{Katona} showed that if $G$ is an intersecting
$r$-graph then $e(G)\leq |\partial(G)|$. This is  known as the {\em intersection shadow theorem}.
 More recently, Frankl~\cite{Frankl} showed that if $G$ is an $r$-graph that does not contain
a matching of size $s+1$ then $e(G)\leq s|\partial(G)|$. Sometimes it is easier prove the bounds in terms of the shadow size than in terms of $n$ using induction. Instead of Theorems~\ref{main}--\ref{4e} we will prove  bounds on $e(G)$ in terms of $|\partial(G)|$, from which Theorems~\ref{main}--\ref{4e} will follow.

\medskip
Based on our results, we propose the following conjecture, which we will show is equivalent to Kalai's conjecture.

%It is natural to ask the following strengthening of Kalai's conjecture.

\begin{conj} \label{main-question}
Let $r\geq 2,t\geq 1$ be integers. Let $T$ be a tight $r$-tree with $t$ edges. If $G$ is an $r$-graph that
does not contain $T$ then $e(G)\leq \frac{t-1}{r}|\partial(G)|$.
\end{conj}

The lower bound constructions obtained from designs mentioned earlier show that the bound in Conjecture~\ref{main-question}, if true, would be tight.
Since for every $r$-graph $G$ on $n$ vertices one has $|\partial (G)|\leq {n \choose r-1}$ Conjecture~\ref{main-question}
obviously implies Conjecture~\ref{kalai}. {   We will show} in Theorem~\ref{th:equi} 
{   that} Conjecture~\ref{kalai} also implies Conjecture~\ref{main-question}.

\begin{prop}~\label{prop:equi}
Conjecture~\ref{main-question} is equivalent to Kalai's conjecture.
  \end{prop}

\begin{thm}\label{th:equi}
If $T$ is a tight tree then the limit
$$
    \alpha(T) :=  \lim_{n\to \infty} \ex_r(n, T)/{n \choose r-1}
  $$
exists and is equal to its supremum. Moreover,
$$
    \alpha(T)=\sup  \left\{  \frac{ e(G) }{ |\partial (G)|} :  G  \text{ is  a $T$-free $r$-graph}\right\} .
  $$
In particular for $\alpha:=\alpha (T)$  we have $\ex(n, T)\leq \alpha {n \choose r-1}$ and
$e(G)\leq \alpha |\partial (G)|$ for every $n$ and for every $T$-free $r$-graph $G$.
  \end{thm}

% This reveals what Kalai's construction is really about.

Let  $H$ be a $u$-uniform hypergraph on  $v$  vertices ($v\geq u\geq 1$).
An {\em almost disjoint induced packing} of $H$ of  size $m$ on $n$ vertices consists of $v$-subsets of  $[n]$, $V_1,\dots, V_m$, and $m$ copies of $H$ on these sets, $H_1, \dots ,H_m$, such that either $|V_i\cap V_j|<u$  or  $|V_i\cap V_j|=u$, but in the
latter case  $V_i\cap V_j$ is not an edge of any of the $H_k$'s.
So $V_k$ induces $H_k$ in the union $\cup E(H_i)$. Obviously, $m\leq {n\choose u}/ e(H)$.
For the proof  of Theorem~\ref{th:equi} we need a result from~\cite{FF64} about the existence of almost perfect induced packings of subhypergraphs with nearly disjoint vertex sets. We recall it in the form we need. Given $H$ as $n \to \infty$ one has
\begin{equation}\label{eq:FF64}
    \max \,  m = (1+o(1))\, {n\choose u}/ e(H).
  \end{equation}
In fact \eqref{eq:FF64} is an application of the packing result of Frankl and R\"odl~\cite{FR} .
% https://faculty.math.illinois.edu/~z-furedi/PUBS/furedi_frankl_induced-packings.pdf

\begin{lem} \label{lem:equi}
Let $T$ be a tight $r$-tree and suppose that $G$ is a $T$-free $r$-graph.
Then for every $\eps>0$, there exists $n_0=n_0(T,G,\eps)$ such that for all $n>n_0$
 $$\ex(n, T) > \left(\frac{e(G)}{|\partial(G)|} -\eps \right){n \choose r-1}. $$
\end{lem}

\begin{proof}[Proof of Lemma~\ref{lem:equi}]
To get a lower bound we need a construction $F$, a $T$-free $r$-graph on $n$ vertices.
Define $H=\partial (G)$ and apply~\eqref{eq:FF64} (with $u=r-1$) to obtain near optimal
number of copies of $\partial (G)$, $H_1, \dots ,H_m$ with vertex sets
$V_1,\dots,V_m$. Put a copy of $G$, $G_i$, on each $V_i$ such that $\partial (G_i)=H_i$.
The resulting copies of $G$ share no $(r-1)$-shadow and in particular are edge-disjoint.
The union  $F= \cup E(G_i)$ has  $(1-o(1)) (e(G)/|\partial(G)|) \binom{n}{r-1}$ edges and it is $T$-free.
Indeed, $F$ cannot contain a tight tree that moves from one copy of $G_i$ to another.
When we start to build the tree $T=\{ e_1, \dots, e_t\}$ with $e_1\in G_i$ then
all other edges $e_j$ must also belong to $G_i$ so there is no such tree in $F$.
\end{proof}

Note that a similar proof idea was used by Huang and Ma~\cite{HM} to disprove an Erd\H{o}s-S\'os/Verstra\"ete conjecture concerning tight cycles.

\begin{rem} \label{rem:connected}
{\rm It {   follows from the proof of}  Lemma~\ref{lem:equi}
that the lemma still holds if $T$ is replaced with any $r$-graph with 
a {\it connected} $(r-1)$-intersection graph, meaning that the auxiliary graph defined on $E(T)$ where $e,e'\in E(T)$
are adjacent if and only if $|e\cap e'|=r-1$ is connected.
}
\end{rem}

\begin{proof}[Proof of Theorem~\ref{th:equi}]
Define
\begin{eqnarray*}
 \alpha(n,T) &:=&  \ex_r(n, T)/{n \choose r-1},\\
  \beta(n, T)&:=& \max \left\{  \frac{ e(G) }{ |\partial G|} :   \text{ $G$ is a $T$-free $r$-graph on $n$ vertices}\right\}.
  \end{eqnarray*}
Since $\beta(n, T)\leq \beta(n+1,T)$ and $\beta(n,T)\le e(T)-1$ (by~\eqref{tb}) the limit
$\beta=\beta(T)=\lim_{n\to \infty} \beta(n,T)$ exists, is positive, and is equal to its supremum.
Since $\alpha(n,T)\leq \beta(n,T)$ we have $\sup_n \alpha(n,T)\leq \beta$.
The proof of the existence of the limit $\alpha$ can be completed by Lemma~\ref{lem:equi}  showing that for
 every $\eps>0$ taking a $T$-free $r$-graph $G$ with $\frac{e(G)}{|\partial(G)|}> \beta-\eps$ there exists an $n_0$
 such that  $ \alpha(n,T) > \beta -2\eps$ for all $n> n_0$.
\end{proof}

%The proof of Theorem~\ref{FF} is surprisingly short. It uses the so-called {\it weight method}. In this paper, we further explore the weight method
%to obtain an upper bound on the maximum size of an $n$-vertex $T$-free $r$-graph for a family of tight trees that is close to the bound conjectured
%by Kalai. We also solve the conjecture exactly for a few tight $3$-trees that belong to this family.

%%%%%%%%%%%%%%%%%%%%%%%%%%%%%%%%%%%%%%%%%%%%%%%%%%%%

%\begin{thm} \label{two-central-weaker}
%Let $T$ be a tight $3$-tree with $t\geq 5$ edges. Suppose $T$ has a trunk $\{e_1,e_2\}$ of size $2$ such that
%$d_T(e_1\cap e_2)\geq \lfloor\frac{t-1}{3}\rfloor+2$. Let $G$ be an
%$n$-vertex $3$-graph that does not contain $T$.
%Then $e(G)\leq \frac{t-1}{3}|\partial(G)|$. In particular, $e(G)\leq \frac{t-1}{3}\binom{n}{2}$.
%\end{thm}

%%%%%%%%%%%%%%%%%%%%%%%%%%%%%%%%%%%%%%%%%%%%%%%%%%%%%%%

\section{Notation and preliminaries}

Given an $r$-graph $G$ and a subset $D\subseteq V(G)$, we define the {\it link} of $D$ in $G$, denoted
by $L_G(D)$, to be
\[L_G(D)= \{ e\setminus D: e\in E(G), D\subseteq e\}.\]
The {\it degree} of $D$, denoted by $d_G(D)$, is defined to be $|L_G(D)|$; equivalently it is the number of edges of $G$ that contain $D$.
When $G$ is $r$-uniform and $|D|=r-1$, elements of $L_G(D)$ are vertices. In this case,
we also use $N_G(D)$ to denote $L_G(D)$ and call it the {\it co-neighborhood} of $D$ in $G$.
When the context is clear we will drop the subscripts in $L_G(D)$, $N_G(D)$ and $d_G(D)$.
For each $1\leq p\leq r-1$, we define the {\it minimum $p$-degree} of $G$ to be
\[\delta_p(G)=\min\{d_G(D): |D|=p, \quad \mbox{\em and}\quad D\subseteq e \quad\mbox{\em for some}\quad  e\in E(G) \}.\]
Given an $r$-graph $G$, and $D\in \partial(G)$, let $w(D)=\frac{1}{d_G(D)}$.
For each $e\in E(G)$, let
\begin{equation} \label{edge-weight}
w(e)=\sum_{D\in \binom{e}{r-1}}  w(D)= \sum_{D\in \binom{e}{r-1}}\frac{1}{d_G(D)}.
\end{equation}
We call $w$ the {\it default weight function} on $E(G)$ and $\partial(G)$.
The following simple property of the default weight function  is key to the weight method, employed
% by Frankl and F\"uredi
in~\cite{FF} and % by others
in various other works.

\begin{prop} \label{weight}
Let $G$ be an $r$-graph. Let $w$ be the default weight function on $E(G)$ and $\partial(G)$. Then \[\sum_{e\in E(G)} w(e)=|\partial(G)|.\]
\end{prop}
\begin{proof}
By definition, %we have
\[\sum_{e\in E(G)} w(e) =\sum_{e\in E(G)} \sum_{D\in \binom{e}{r-1}} \frac{1}{d_G(D)}
=\sum_{D\in \partial(G)} \sum_{e\in E(G), D\subseteq e} \frac{1}{d_G(D)}
=\sum_{D\in \partial(G)} 1 =|\partial(G)|.\]
\end{proof}

An $r$-graph $G$ is called {\it $r$-partite} if $V(G)$ can be partitioned into $r$ sets $A_1,\dots, A_r$
such that every edge of $G$ contains one vertex from each $A_i$. We call $(A_1,\dots, A_r)$ an {\it $r$-partition}
of $G$. Equivalently, we say that an $r$-graph $G$ is {\it $r$-colorable} if $G$ if there exists a vertex coloring of $G$ with $r$ colors such
that each edge uses all $r$ colors; we call such a coloring a {\it proper  $r$-coloring} of $G$.
The following proposition follows by  induction on the
number of edges in $T$.
\begin{prop} \label{tight-tree}
Let $r\geq 2$. Every tight $r$-tree $T$ has a unique $r$-partition.   \qed
%is $r$-partite with an $r$-partition that is unique up to permutation of parts. \qed
\end{prop}

Given $r$-graphs $G$ and $H$, an {\it embedding} of $H$ into $G$ is an injection $f: V(H)\to V(G)$ such that
for each $e\in E(H)$, $f(e)\in E(G)$.

\begin{prop}{\bf (Color-preserving embedding)} \label{colored-mapping}
Let $T$ be a tight $r$-tree with $t$ edges. Let $\phii$ be a proper $r$-coloring of $T$.
Let $G$ be an $r$-partite graph with $\delta_{r-1}(G)\geq t$, where $(A_1,\cdots, A_r)$ is an $r$-partition of $G$.
Then there exists an embedding $f$ of $T$ into $G$ such that for each $u\in V(T)$ $f(u)\in A_{\phii(u)}$.
\end{prop}
\begin{proof}
We use induction on $t$. The base step is trivial. Now, suppose $t\geq 2$.
Let $e_1,\dots, e_t$ be an ordering of the edges of $T$ that satisfies \eqref{tree-definition}.
Let $T'=T\setminus e_t$. Then $T'$ is a tight $r$-tree with $t-1$ edges.
By the induction hypothesis, there exists an embedding $f$ of $T'$ into $G$ such that for each $u\in V(T')$,
$f(u)\in A_{\phii(u)}$. Let $D=e_t\cap e_{\alpha(t)}$ and let $v$ be the unique vertex in $e_t\setminus e_{\alpha(t)}=
{   V(T)\setminus V(T')}$.  Then $e_t=D\cup \{v\}$.
Since $f(D)$ is an $(r-1)$-set contained in $f(e_{t-1})$ and $\delta_{r-1}(G)\geq t$,
$d_G(f(D))\geq t$. So there are at least $t$ edges of $G$ containing $f(D)$, at most $|V(T')|-(r-1)=t-1$ of which
contain a vertex of $f(T')$. Hence there exists an edge $e$ in $G$ that contains $f(D)$ and a vertex $z$ outside
$f(T')$. We extend $f$ by letting $f(v)=z$. Now $f$ is  an embedding of $T$ into $G$.

It remains to show that $z\in A_{\phii(v)}$. By permuting colors if needed, we may assume that $\phii(v)=r$.
Since $D\cup \{v\}\in E(T)$ and $\phii$ is proper, the colors used in $D$ are $1,\dots, r-1$. By our assumption, vertices in $f(D)$ lie in
$A_1,\dots, A_{r-1}$, respectively, which implies $z\in A_r$.
\end{proof}
The following proposition is folklore. We include a proof for completeness,

\begin{prop} \label{min-codegree} Let $r\geq 2$ and $q\geq 1$ be integers and
let $G$ be an $r$-graph with $e(G)>q|\partial(G)|$. Then $G$ contains a subgraph $G'$ with $\delta_{r-1}(G')\geq q+1$
{   and
\begin{equation}\label{1208}
e(G')>q|\partial(G')|.
\end{equation}}
\end{prop}
\begin{proof}
{   Among subgraphs $G'$ of $G$ satisfying~(\ref{1208}), choose one with the fewest edges. We claim that
$\delta_{r-1}(G')\geq q+1$.
Indeed, if there is $D\in \partial(G')$ that is contained in  at most  $q$
edges of $G'$, then the $r$-graph $G''$ obtained from $G'$ by deleting all edges containing $D$ again satisfies~(\ref{1208}), but has fewer edges than $G'$, a contradiction.}
%Starting from $G$, whenever there exists $D\in \partial(G)$ that is contained in at least one but at most  $q$
%edges of the remaining $r$-graph, we remove  the edges of this $r$-graph that contain $D$. Let $G'$ be
%the $r$-graph we obtain at the end of the process. Since at most $q|\partial(G)|<e(G)$ edges are removed in the process,
%$G'$ is nonempty. By the stopping rule, $\delta_{r-1}(G')\geq q+1$.
\end{proof}

Another useful folklore fact is:

\begin{prop} \label{shad} Let $\alpha$ be a positive real, $r\geq 3$ be an integer and $G$
be an $r$-graph with $e(G)>\frac{\alpha}{r} |\partial(G)|$. Then there is $v\in V(G)$ such that
the link $G_1:=L_G(\{v\})$ satisfies
$$e(G_1)>\frac{\alpha}{r-1} |\partial(G_1)|.
$$
\end{prop}
\begin{proof} Suppose that $|L_G(\{v\})|\leq \frac{\alpha}{r-1}|\partial (L_G(\{v\})|$ for each $v\in V(G)$. Then
$$r\cdot e(G)=\sum_{v\in V(G)}d_G(v)=\sum_{v\in V(G)}|L_G(\{v\})|\leq  \frac{\alpha}{r-1}\sum_{v\in V(G)}|\partial (L_G(\{v\})|.
$$
Since each edge $f\in \partial(G)$ contributes $r-1$ to $\sum_{v\in V(G)}|\partial (L_G(\{v\})|$ ($1$ to the link of each its vertex),
this proves the proposition.
\end{proof}

We also need the following fact used in~\cite{FF}.
\begin{prop} \label{arrange}
Let $r$ be a positive integer.
Let $d_1\leq d_2,\dots \leq d_r$ be positive reals.
If $\sum_{i=1}^r \frac{1}{d_i}=s$,
then for each $i\in [r]$, $d_i\geq \frac{i}{s}$.
\end{prop}
\begin{proof}
For each $i\in [r]$, since $\frac{1}{d_1}\geq \dots \geq \frac{1}{d_i}$,
we have $\frac{i}{d_i}\leq \sum_{j=1}^i \frac{1}{d_j}\leq s$. So, $d_i\geq \frac{i}{s}$.
\end{proof}
%%%%%%%%%%%%%%%%%%%%%%%%%%%%%%%%%%%%%%%%%%%%%%%%%%%%%%%%

\section{Proof of Theorem~\ref{main} {    on trees with} bounded trunks}\label{sec:main}

As discussed in the introduction, we prove the following stronger version of Theorem~\ref{main}.

\medskip \noindent
{\bf Theorem~\ref{main}$'$.}
{\em 
Let $n,r,t,c$ be positive integers, where $n\geq r\geq 2$ and $t\geq c\geq 1$.
Let $a(r,c)=(r^r+1-\frac{1}{r})(c-1)$.
Let $T$ be a tight $r$-tree with $t$ edges and $c(T)\leq c$. % that has trunk $T'$ size at most $c$.
If $G$ is an $r$-graph that
does not contain $T$ then }
\begin{equation}\label{eq:main2}
e(G)\leq  \left(\frac{t-1}{r}+a(r,c) \right)|\partial (G)|.
\end{equation}

\medskip
% {\bf Proof.}
\begin{proof}[Proof of Theorem~\ref{main} $'$]
Suppose $T$ is a tight $r$-tree with $t$ edges and
$c(T)=c$. Let $G$ be an $n$-vertex $r$-graph with 
$e(G)> (\frac{t-1}{r}+a(r,c))|\partial (G)|$.
We show that $G$ contains $T$. For convenience, let
\[\gamma=\frac{t-1}{r}+a(r,c)-r^r(c-1)=\frac{t-1}{r}+(1-\frac{1}{r})(c-1).\]
Then
\[e(G)>(\gamma+r^r(c-1))|\partial(G)|.\]
Let $w$ be the default weight function on $E(G)$ and $\partial(G)$.
By Proposition~\ref{weight},  $\sum_{e\in E(G)} w(e) =|\partial(G)|$.
Let
\[H=\{e\in E(G): w(e)\geq \frac{1}{\gamma}\} \mbox { and } L=\{e\in E(G): w(e)<\frac{1}{\gamma}\}.\]
By the definition of $H$,
\[\frac{1}{\gamma}e(H)\leq \sum_{e\in H} w(e) \leq \sum_{e\in G} w(e)=|\partial (G)|.\]
Hence $e(H)\leq \gamma|\partial(G)|$.
Since $e(G)>(\gamma+r^r (c-1))|\partial(G)|$, we have
\[e(L)> r^r(c-1)|\partial(G)|.\]
By averaging, $L$ contains an $r$-partite subgraph $L_1$ with
\begin{equation}\label{L1-size}
e(L_1)\geq \frac{r!}{r^r} e(L)> \frac{r!}{r^r} r^r (c-1)|\partial(G)|\geq r!(c-1)|\partial(G)|.
\end{equation}
Let $(A_1,\dots, A_r)$ be an $r$-partition of $L_1$. Let $e\in E(L_1)$.
Let $\sigma$ be a permutation of $[r]$ such that
\[d_G(e\setminus A_{\sigma(1)})\leq \dots \leq d_G(e\setminus A_{\sigma(r)}).\]
We let $\pi(e)=(\sigma(1),\dots, \sigma(r))$ and refer to it as the {\it pattern} of $e$.
Since there are $r!$ different permutations of $[r]$, by the pigeonhole principle,
some  $\lceil e(L_1)/r!\rceil$ edges $e$ of $L_1$ have the same pattern $\pi(e)$. Let
$L_2$ be the subgraph of $L_1$ consisting of these edges. By \eqref{L1-size},
\[e(L_2)\geq \frac{e(L_1)}{r!}> (c-1)|\partial(G)|.\]
By Lemma~\ref{min-codegree}, $L_2$ contains a subgraph $L^*_2$ such that
\[\delta_{r-1}(L^*_2)\geq c.\]
Recall that all edges in $L^*_2\subseteq L_1$ have the same pattern.
By permuting indices if needed, we may assume that  $ \pi(e)=(1,2,\dots, r)$
for each $ e\in L^*_2$.
By our assumption,
\begin{equation} \label{decreasing-sequence}
 d_G(e\setminus A_1)\leq \cdots \leq d_G(e\setminus A_r)\qquad \forall e\in L^*_2.
\end{equation}
Also, by the definition of $L$,
\[ w(e)=\sum_{i=1}^r \frac{1}{d_G(e\setminus A_i)}<\frac{1}{\gamma} \qquad \forall e\in L^*_2\subseteq L.\]
By Lemma~\ref{arrange} and \eqref{decreasing-sequence}, we have
\begin{equation} \label{codegree-lower}
  d_G(e\setminus A_i) > i\gamma \qquad \forall e\in L^*_2\,\, \qquad \forall i\in [r].
\end{equation}

Now consider a trunk $T'$ of $T$ %, which by our assumption, is a tight $r$-tree
 with $c$ edges.
%Also observe that b
By the definition of a trunk, if $E'$ is any subset of $E(T)\setminus E(T')$ then
$T'\cup E'$ is a tight tree with $c+|E'|$ edges.
By Proposition~\ref{tight-tree}, $T'$ is $r$-partite. Let $(B_1,\dots, B_r)$ be an $r$-partition of $T'$.
For each $e\in E(T)\setminus E(T')$, by definition, there exists $\alpha(e)\in E(T')$ such that
$|e\cap\alpha(e)|=r-1$. Thus, $e\cap \alpha(e)=\alpha(e)\setminus B_i$ for some unique $i\in [r]$.
For each $i\in [r]$, let
\[E_i=\{e\in E(T)\setminus E(T'): e\cap \alpha(e)=\alpha(e)\setminus B_i\}.\]
By permuting the subscripts in the $r$-partition $(B_1,\dots, B_r)$ of $T'$ if needed, we may assume
that
\[|E_1|\leq \dots \leq |E_r|.\]
Since $\sum_{i=1}^r |E_i| = t-c$, this implies
\begin{equation}\label{partial-sum}
 |E_1|+\dots+|E_i|\leq \left\lfloor\frac{i(t-c)}{r}\right\rfloor\qquad \forall i\in [r].
\end{equation}
Since $e(T')=c$, $\delta_{r-1}(L^*_2)\geq c$, $(A_1,\dots, A_r)$ is an $r$-partition of $L^*_2$
and $(B_1,\dots, B_r)$ is an $r$-partition of $T'$,
by Proposition~\ref{colored-mapping}, there exists an embedding $h$ of $T'$ into $L^*_2$ such that
for each $i\in [r]$ every vertex in $B_i$ of $T'$ is mapped into $A_i$. Now consider the edges in $E_1$.
By the definition of $E_1$, %our earlier discussions,
for each $ e\in E_1$ there is $\alpha(e)\in E(T')$ such that  $e\cap \alpha(e)=\alpha(e)\setminus B_1$
and $h(\alpha(e\setminus B_1))=h(\alpha(e))\setminus A_1$. Since $h(\alpha(e))\in L^*_2$,
by \eqref{codegree-lower},
\begin{equation} \label{E1-codegree}
 d_G(h(\alpha(e)\setminus A_1))\geq \lfloor \gamma \rfloor+1\qquad \forall e\in E_1.
\end{equation}
Since $T'\cup E_1$ is a tight tree with
$$|E_1|+c\leq \lfloor\frac{t-c}{r}\rfloor+c
=\lfloor \frac{t-1}{r}+(1-\frac{1}{r})(c-1)\rfloor+1=\lfloor \gamma\rfloor+1$$ edges, and $h$ is an embedding
of $T'$ into $G$, \eqref{E1-codegree} ensures that we can greedily extend $h$ to an embedding of $T'\cup E_1$ into $G$.
In general, let $i\in [r]\setminus \{1\}$ and suppose that we have extended $h$ to an embedding of $T'\cup E_1\cup\dots
\cup E_{i-1}$ into $G$.
By  the definition of $E_i$, %our earlier discussion,
for each $ e\in E_i$ there is $\alpha(e)\in T'$ such that
$ e\cap \alpha(e)=\alpha(e)\setminus B_i$ and $h(e\cap \alpha(e))=h(\alpha(e))\setminus A_i$.
By \eqref{codegree-lower},
\begin{equation} \label{Ei+1-codegree}
 d_G(h(e\cap \alpha(e))\geq \lfloor i\gamma\rfloor+1\qquad \forall e\in E_i.
\end{equation}
Since $T'\cup E_1\cup \dots \cup E_i$ is a tight tree with
$$c+|E_1|+\dots +|E_i|
\leq \lfloor\frac{i(t-c)}{r}\rfloor+c\leq\lfloor i\gamma\rfloor+1$$ edges, and $h$ is already an embedding of $T'\cup E_1\cup \dots \cup E_{i-1}$ into $G$,
\eqref{Ei+1-codegree} ensures that we can greedily extend $h$ further to an embedding
{    of} $T'\cup E_1\cup \dots \cup
E_{i}$ into $G$. Hence we can find an embedding of $T$ into $G$. %This completes the proof of Theorem~\ref{main}.
%%%\qed
\end{proof}
%%%%%%%%%%%%%%%%%%%%%%%%%%%%%%%%%%%%%%%%%%%%%%%%%%%%

\section{Proof of Theorem~\ref{4e} {   on} trees with four edges}\label{sec:4e}

%{\bf Proof of Theorem~\ref{p4}:}
Again, we are proving the shadow version of the theorem:

\medskip \noindent
{\bf Theorem~\ref{4e}$'$.}
{\em  Let $n\geq r\geq 2$ be  integers and $T$ be a tight $r$-tree with $t\leq 4$ edges.
If $G$ is an $r$-graph that
does not contain $T$ then $e(G)\leq \frac{t-1}{r}  |\partial (G)|$.}

\medskip
We start from a partial case of such $T$, the $3$-uniform tight path $P^3_4$ with $4$ edges.
The case of the path $P^3_5$ is still unsolved (to our knowledge).

\begin{lem}\label{P34} Let $n\geq 5$ and
$G$ be an $n$-vertex $3$-graph containing no tight path $P_4^3$ with four edges.
Then $e(G) \leq |\partial(G)| $.
\end{lem}

Observe that for $3$-graphs Lemma~\ref{P34} is stronger than Katona's intersecting 
shadow theorem, since an intersecting $3$-graph must be $P^3_4$-free.
There are many nearly extremal families  with very different structures for Lemma~\ref{P34}
besides the ones obtained from Steiner systems $S(n,5,2)$. Here we mention just two.
First, one can observe that the Erd\H{o}s-Ko-Rado family $G:= \{ g\in
{[n]\choose 3}: 1\in g\}$ is $P_4^3$-free with
$$|\partial (G)|= {n \choose 2}= \frac{n}{n-2} {n-1\choose 2}= \frac{n}{n-2} e(G).
$$
Second, for $n\equiv 0 \mod 3$ one can take a tournament  $\overrightarrow{D}$ on $n/3$ vertices and a partition of $[n]$ into
triples $V_1, V_2, \dots, V_{n/3}$ and define the  $P_4^3$-free triple system as
$$G:= \left\{ g\in {[n]\choose 3}: \text{ for some }\overrightarrow{ij}\in E(\overrightarrow{D}) \text{ one
has } |V_i\cap g|=2, \, |V_j\cap g|=1  \right\}.
  $$
Then we have $|\partial (G)|/e(G) = {n \choose 2}/ 9{n/3 \choose 2}= (n-1)/(n-3)$.

\begin{proof}[Proof of Lemma~\ref{P34}]  Suppose $G$ is an $n$-vertex $3$-graph with the fewest edges such that
\begin{equation}\label{e0}
e(G) > |\partial (G)| \;\mbox{\em and $G$ contains no $P^3_4$.}
\end{equation}
By Proposition~\ref{min-codegree} and the minimality of $G$,
\begin{equation}\label{e1}
\delta_2(G)\geq 2.
\end{equation}
 %for every $e \in H$, the profile is $(a,b,c)$ with $a,b,c \geq 2$.

Let $w$ be the default weight function on $G$ and $\partial (G)$. Since $\sum_{e\in G}w(e)=|\partial (G)|<e(G)$, by~\eqref{e0}, $G$ has
an edge $e_0=abc$ with
\begin{equation}\label{e2}
w(e_0)=\frac{1}{d(ab)}+\frac{1}{d(ac)}+\frac{1}{d(bc)}<1.
\end{equation}
We may assume $d(ab)\leq d(ac)\leq d(bc)$. Similarly to Proposition~\ref{arrange}, in order~\eqref{e2} to hold, we need
\begin{equation}\label{e3}
d(ac)\geq 3\quad\mbox{and}\quad d(bc)\geq 4.
\end{equation}
By~\eqref{e1} and~\eqref{e2}, we can greedily choose distinct $a',b',c'\in V(G)-\{a,b,c\}$ so that
$abc',acb',bca'\in G$.

We claim that
\begin{equation}\label{e4}
ab'b,ac'c\in G.
\end{equation}
Indeed, by~\eqref{e1} $G$ has an edge $ab'x$ for some $x\neq c$. If $x\notin \{b,a'\}$, then $G$ has a tight
$4$-path $a'bcab'x$, a contradiction to~\eqref{e0}. So suppose $x=a'$. By~\eqref{e3},
$G$ has an edge $bcy$ for some
$y\notin \{a,a',b'\}$. Then $G$ has a tight
$4$-path $ybcab'a'$, again a contradiction to~\eqref{e0}. Thus $ab'b\in G$. Similarly, $ac'c\in G$, and~\eqref{e4} holds.

Next we similarly show that
\begin{equation}\label{e5}
a'ba,a'ca\in G.
\end{equation}
Indeed, by~\eqref{e1} $G$ has an edge $a'bx$ for some $x\neq c$. If $x\notin \{a,b'\}$, then $G$ has a tight
$4$-path $b'acba'x$. Suppose $x=b'$. Then by~\eqref{e4},
 $G$ has a tight $4$-path $b'a'bcac'$, again a contradiction to~\eqref{e0}.
Thus $a'ba\in G$. Similarly, $a'ca\in G$, and~\eqref{e5} holds.

Together,~\eqref{e4} and~\eqref{e5} imply that $d_G(ab)\geq 4$ and $d_G(ac)\geq 4$. So, the proof of~\eqref{e4} yields similarly
that $c'bc,b'cb\in G$. If the degree of each of $a'a,a'b,a'c$ is $2$, then the $3$-graph $G_2=G\setminus \{a'ab,a'ac,a'bc\}$
has $|G|-3$ edges and $|\partial(G_2)|= |\partial (G)|-3$, a contradiction to the minimality of $G$. Thus  we may assume
that $G$ has an edge $a'ax$, where $x\notin \{b,c\}$.
By the symmetry between $b'$ and $c'$, we may assume $x\neq b'$.
Then $G$ has a tight
$4$-path $xa'abcb'$.
\end{proof}

Now we are ready to prove Theorem~\ref{4e}$'$.

\begin{proof}[Proof of Theorem~\ref{4e} $'$] 
 We use induction on $r$.

%%% We use double  induction: first on $r$ and second on $e(G)$. {   Why induction on $e(G)$??}0

{\bf Base Step.} $r=2$. In this case, $\partial (G)=V(G)$. For $t\leq 3$ the statement is trivial. Let $t=4$.
There are three non-isomorphic (graph) trees with four edges:
the path $P_4=v_0v_1v_2v_3v_4$ with $4$ edges, the star $S_4$ with center $v_0$ and leaves $v_1,v_2,v_3,v_4$, and
the tree $F_4$ obtained from the star $S_4$ by replacing edge $v_0v_4$ with edge $v_4v_3$.
So we want to show that for every graph $G$
\begin{equation}\label{graphs}
\mbox{\em if $e(G)>\frac{3}{2}|V(G)|$ then $G$ contains each of $P_4,S_4$ and $F_4$.}
\end{equation}
The case of $T=P_4$ is a special case of the Erd\H{o}s-Gallai Theorem~\cite{ErdGal59}. The other two possibilities also
are known from the literature, but we give a short proof. Consider a counterexample $G$ to~\eqref{graphs} with the fewest vertices.
By Proposition~\ref{min-codegree}, $\delta(G)\geq 2$. Since $\sum_{a\in V(G)}d(a)=2e(G)>3|V(G)|$, there is $a\in V(G)$ with $s\geq 4$
neighbors, say, $b_1,\ldots,b_s$. In particular, $G$ contains $S_4$ with center $a$. Since $\delta(G)\geq 2$, $b_1$ has a neighbor $b\neq a$.
So we can embed $F_4$ into $G[\{a,b,b_1,\ldots,b_4\}]$ by sending $v_0$ to $a$, $v_3$ to $b_1$, $v_4$ to $b$, and $v_1$ and $v_2$ to
two vertices in $\{b_2,b_3,b_4\}-b$. Thus~\eqref{graphs} holds.

\medskip
{\bf Induction Step.} Suppose $r\geq 3$, the theorem holds for all $r'<r$, $T$ is a tight $r$-tree, and $G$ is
an $r$-graph with $e(G)> \frac{t-1}{r}  |\partial (G)|$.

{\bf Case 1:} $T$ has a vertex $v$ belonging to all edges. Let $T_1$ be the link $L_T(\{v\})$ of $v$.
It is a tight $(r-1)$-tree with $t$ edges.
By Proposition~\ref{shad},
there is $a\in V(G)$ such that
the link $G_1:=L_G(\{a\})$ satisfies
$e(G_1)>\frac{t-1}{r-1} |\partial(G_1)|.$ By the induction assumption, there is an embedding $\phii$ of
$T_1$ into $G_1$. Then by letting $\phii(v)=a$ we obtain an embedding of $T$ into $G$.

{\bf Case 2:} $T$ has no vertex  belonging to all edges. By the definition of a tight $r$-tree, this is possible only
if $t=4$, $r=3$ and $T=P^3_4$. In this case, we are done by Lemma~\ref{P34}.
\end{proof}
%%%%%%%%%%%%%%%%%%%%%%%%%%%%%%%%%%%%%%%%%%%%%%%%

\section{Concluding remarks}
\parindent=0pt
\begin{itemize}

\item Theorem~\ref{th:equi} shows that  some shadow theorems in the literature are not really stronger than their nonshadow versions.
In particular, this is the case  whenever the forbidden $r$-graph $T$  has a connected $(r-1)$-intersection graph (see Remark~\ref{rem:connected}).

\item It would be interesting to decide if Lemma~\ref{lem:equi} holds for other $r$-graphs besides tight trees and
also for which $r$-graphs $T$ $\lim_{n\to \infty} \ex_r(n,T)/\binom{n}{r-1}$ exists.
In particular, we ask if $\lim_{n\to\infty} \ex_r(n,T)/\binom{n}{r-1}$ exists for each $r$-uniform tree $T$,
where an $r$-graph is a {\it tree} if it is a subgraph of a tight tree. 
This question is not even solved when $r=2$ and $T$ is a graph forest, see, e.g.,~\cite{LLP}. 
See~\cite{FJ-tree} and~\cite{KMV-tree} for recent results on the Tur\'an numbers of  some large families of $r$-uniform trees.
%This question is not even solved when $r=2$ and $T$ is a graph forests, see, e.g.,~\cite{LLP}. 

Note that even for trees,  if the limits $\alpha(T)$ and $\beta(T)$ exist they need not be equal.
(See the proof of  Theorem~\ref{th:equi} for the definition of $\alpha(T)$ and $\beta(T)$.)
Consider a linear path $P=P_4^r$ of length four, $E(P):= \{  \{1,2,\dots, r \},   \{r,r+1,\dots, 2r-1 \},  \{2r-1,2r,\dots, 3r-2 \}, \{3r-2,3r-1,\dots, 4r-3 \}\}$.  It is known~\cite{FJ-tree, KMV-tree} that $\ex(n,T)={n-1 \choose r-1}+ {n-3\choose r-2}+ \eps(n,r)$   for $n> n_0(r)$ and $r\geq 3$, where $\eps(n,r)=0$ except for $r=3$, when it is $0$, $1$ or $2$.   
So we have $\alpha(P)=1$. 
On the other hand, the complete $r$-graph $G$ on $4r-4$ vertices avoids $P_4^r$ and 
  $e(G)/|\partial G|= {4r-4 \choose r}/{4r-4 \choose r-1}= (3r-3)/r  \leq \beta(P)$. 
Consequently, $0 < \alpha(P)< \beta(P)$ for $r\geq 3$.
(Actually, a linear path of length 3 is also an appropriate example).

In the case $r=2$ consider $T=kP_2$, a disjoint union of $k$ paths of length $2$ on $3k$ vertices. Gorgol~\cite{Gorgol} showed that $\alpha(kP_2)= k-1/2$ while considering the complete graph on $3k-1$ vertices we get $\beta(kP_2)\geq (3k-2)/2$.
Moreover,  the Erd\H{o}s-Gallai   Theorem implies that here equality holds here.

\item Recent substantial work by Keller and Lifshitz~\cite{KL} {    studies} the Tur\'an number of some $r$-graphs $F$ with small core. However their 
{\em junta method for hypergraphs}
does not seem to apply here, since it seems to require that $r \gg |C|$ where $C$ is
{   the set of the  vertices of $F$ of degree at least  $2$.}
% a set such that the vertices of $F$ outside $C$ have degree one. 

\item A  direction we will continue to pursue is to reduce the error term $a(r,c)$ in the coefficient in Theorem~\ref{main}. We have some nontrivial improvements.
For example, in the first unsolved case, that is, when $T$ is a $3$-uniform tight tree with $c(T)=2$, we have a proof that $a(3,2)\leq 1/3$. Thus we have $\beta(T)\leq t/3$ and $\ex_3(n,T)\leq (t/3)\binom{n}{2}$. 
\end{itemize}

\paragraph{Acknowledgements.}
This research was partly conducted during an American Institute of Mathematics Structured Quartet Research Ensembles workshop.
The authors gratefully acknowledge the support of AIM.

%%%%%%%%%%%%%%%%%%%%%%%%%%%%%%%%%%%%%%%%%%%%%%%%
%\newpage

{\small

\begin{tabular}{ll}
\begin{tabular}{l}
{\sc Zolt\'an F\" uredi} \\
Alfr\' ed R\' enyi Institute of Mathematics \\
Hungarian Academy of Sciences \\
Re\'{a}ltanoda utca 13-15 \\
H-1053, Budapest, Hungary \\
E-mail:  \texttt{zfuredi@gmail.com}.
\end{tabular}
& \begin{tabular}{l}
{\sc Tao Jiang} \\
Department of Mathematics \\ Miami University \\ Oxford, OH 45056, USA. \\ E-mail: \texttt{jiangt@miamioh.edu}. \\
$\mbox{ }$
\end{tabular} \\ \\
\begin{tabular}{l}
{\sc Alexandr Kostochka} \\
University of Illinois at Urbana-Champaign \\
Urbana, IL 61801, USA \\
and Sobolev Institute of Mathematics \\
Novosibirsk 630090, Russia. \\
E-mail: \texttt {kostochk@math.uiuc.edu}.
\end{tabular} & \begin{tabular}{l}
{\sc Dhruv Mubayi} \\
Department of Mathematics, Statistics \\
and Computer Science \\
University of Illinois at Chicago \\
Chicago, IL 60607, USA. \\
\texttt{E-mail: mubayi@uic.edu}.
\end{tabular} \\ \\
\begin{tabular}{l}
{\sc Jacques Verstra\"ete} \\
Department of Mathematics \\
University of California at San Diego \\
9500 Gilman Drive, La Jolla, California 92093-0112, USA. \\
E-mail: {\tt jverstra@math.ucsd.edu.}
\end{tabular}
\end{tabular}
}

\end{document}